\documentclass[a4paper,11pt,twoside]{amsart}

\usepackage[latin1]{inputenc}
\usepackage{amsmath, amsfonts, amssymb,amsthm}
\usepackage[mathscr]{eucal} 
\usepackage[pdftex]{graphicx}

\usepackage[T1]{fontenc}
\usepackage[sc]{mathpazo}
\usepackage{enumerate}

\newcommand{\N}{\mathbb{N}}

\newcommand{\R}{\mathbb{R}}
\newcommand{\C}{\mathbb{C}}
\newcommand{\s}{\mathbb{S}}
\newcommand{\h}{\mathbb{H}}
\newcommand{\E}{\mathbb{E}}
\newcommand{\Sl}{\mathrm{Sl}_2(\R)}

\newcommand{\Nil}{\mathrm{Nil}_3}

\newcommand{\pRe}{\mathrm{Re}}

\newcommand{\df}{\,\mathrm{d}}

\newcommand{\prodesc}[2]{\left\langle #1, #2 \right \rangle}
\newcommand{\abs}[1]{\left\lvert #1 \right\rvert}

\newtheorem{theorem}{Theorem}
\newtheorem*{theorem*}{Theorem}
\newtheorem{proposition}{Proposition}
\newtheorem{corollary}{Corollary}
\newtheorem{lemma}{Lemma}

\theoremstyle{definition}

\theoremstyle{remark}
  \newtheorem{remark}{Remark}

\numberwithin{equation}{section}

\title[New examples of CMC surfaces in $\s^2 \times \R$ and $\h^2 \times \R$]{New examples of constant mean curvature surfaces in $\s^2 \times \R$ and $\h^2 \times \R$}

\author{José M. Manzano}
\address{Departamento de Geometr\'{\i}a  y Topolog\'{\i}a \\
Universidad de Granada \\
18071 Granada, SPAIN} 
\email{jmmanzano@ugr.es}

\author{Francisco Torralbo}
\address{Departamento de Geometr\'{\i}a  y Topolog\'{\i}a \\
Universidad de Granada \\
18071 Granada, SPAIN} 
\email{ftorralbo@ugr.es}

\thanks{Research partially supported by the MCyT-Feder research project MTM2007-61775, the Junta Andalucía Grants P06-FQM-01642 and P09-FQM-4496 and the \textsc{genil} research project PYR-2010-21}



\begin{document}

\begin{abstract}
We construct non-zero constant mean curvature $H$ surfaces in the product spaces $\s^2 \times \R$ and $\h^2\times \R$ by using suitable conjugate Plateau constructions. The resulting surfaces are complete, have bounded height, and are invariant under a discrete group of horizontal translations. A $1$-parameter family of unduloid-type surfaces is produced in $\s^2\times\R$ for any $H>0$ (some of which are compact), and in $\h^2\times\R$ for any $H>1/2$ (which are shown to be properly embedded bigraphs). Finally, we give a different construction in $\h^2 \times \R$ for $H=1/2$ giving surfaces with the symmetries of a tessellation of $\h^2$ by regular polygons. 
\end{abstract}

\maketitle

\section{Introduction}

In 1970, Lawson~\cite{Lawson70} established a celebrated correspondence between simply connected minimal surfaces in a space form $M^3(\kappa)$ (with constant curvature $\kappa$) and constant mean curvature (\textsc{cmc}) $H$ surfaces  in the space $M^3(\kappa - H^2)$. This result motivated the construction of two doubly periodic constant mean curvature one surfaces in the Euclidean 3-space. The procedure used to obtain such examples is known as the \emph{conjugate Plateau construction} and has become a fruitful method to obtain constant mean curvature surfaces in space forms (e.g., see \cite{KPS88, Karcher89, GB93, Polthier94}). We summarize the steps of this construction as follows:
\begin{enumerate}
	\item Solve the Plateau problem in a geodesic polygon in $M^3(\kappa)$.
	\item Consider the \emph{conjugate} \textsc{cmc} $H$  surface in $M^3(\kappa - H^2)$, whose boundary lies on some planes of symmetry, since the initial surface is bounded by geodesic curves (cf.~\cite[Section 1]{Karcher89}).
	\item Reflect the resulting surface across its edges to get a complete constant mean curvature $H$ surface in $M^3(\kappa - H^2)$.
 \end{enumerate} 

The key property of this method is that a \emph{geodesic curvature line in the initial surface becomes a planar line of symmetry in the conjugate one}. This is crucial in order to extend by reflection the conjugate piece to a complete constant mean curvature surface. Hence, it is important to cleverly choose the appropriate geodesic polygon once the desired symmetries in the target surface have been fixed. Sometimes it is not possible to explicitly determine the conjugate surface and continuity or degree arguments come in handy (cf.~\cite{KPS88, Polthier94}). 

Nevertheless, one of the main drawbacks of the conjugate Plateau construction is that the produced surfaces are hardly ever known to be embedded, since the correspondence consists in integrating geometric data. Hence, embeddedness has to be proven separately and it becomes a difficult task. Some useful results are the Krust's conjugate graph theorem~\cite[vol.~I, p.~118]{DHKW} and its generalizations to $M\times\R$ ($M$ being a non-positive constant Gaussian curvature surface, see~\cite{HST}), and to homogeneous Riemannian $3$-manifolds (see~\cite{CH13}). These results ensure that the conjugate surface is embedded provided that the initial surface is a graph over a convex domain (see~\cite{Karcher89} for several applications). In most of the cases (e.g., the double periodic \textsc{cmc} surface of $\R^3$ contained in a slab constructed by Lawson~\cite{Lawson70} and generalized by Karcher~\cite{Karcher89}, see also~\cite{GB93} and \cite[\S 4]{GB2005}), the examples are broadly supposed to be embedded but this has hitherto lacked of a proof as far as we know. 

In the last few years the study of constant mean curvature surfaces in the homogeneous Riemannian $3$-manifolds has become an active research topic (see, for example, \cite{DHM09} for a survey on recent results). Daniel~\cite{Daniel07} established a Lawson-type correspondence between constant mean curvature surfaces in homogeneous Riemannian 3-manifolds with isometry group of dimension $4$ (see Section~\ref{subsec:Daniel-correspondence}) that gives rise to an extension of the conjugate Plateau construction to this class of 3-manifolds. It also generalizes the correspondence by Hauswirth, Sa Earp and Toubiana \cite{HST} between minimal surfaces in the product case, where some authors have made their contributions~\cite{MR,R,MRR}.

This paper has a double aim: on the one hand to extend the conjugate Plateau construction to the homogeneous Riemannian 3-manifolds by using the Daniel correspondence and, on the other hand, to obtain, applying this procedure, new constant mean curvature surfaces in the product spaces $\s^2\times \R$ and $\h^2 \times \R$, where $\s^2$ and $\h^2$ stand for the sphere and the hyperbolic plane with curvature one and minus one, respectively.

Section~\ref{sec:preliminaries} introduces the Daniel correspondence, and studies how curves and symmetries in the corresponding surfaces are related. We will realize that it is important to deal with polygons made of vertical and horizontal geodesics, and also that the phase angle $\theta$ of the Daniel correspondence is equal to $\pi/2$, because this will be the case where we will be able to handle the geometry of the conjugate surface. It will turn out that the target space must be a Riemannian product manifold, which makes sense since they are the only $\mathbb E(\kappa,\tau)$-spaces admitting totally geodesic surfaces~\cite{ST09} and thus enabling mirror symmetries (cf.\ Lemmas~\ref{lm:simetrias-productos} and~\ref{lm:simetrias-correspondencia-daniel}). All these features make the choice of the initial geodesic polygon more rigid, so the needed arguments become more subtle than in the Lawson setting. We will finish Section~\ref{sec:preliminaries} by discussing the smooth extension of the Plateau solution by reflection across its border. 

In Section~\ref{sec:homogeneous-spaces} we include a brief description of those homogeneous spaces involved in the construction, as well as their needed properties.

Section~\ref{sec:spherical-helicoids} is devoted to the first non-trivial examples of conjugate Plateau construction, showing that the $1$-parameter family of spherical helicoids in the Berger spheres (which are surfaces ruled by horizontal geodesic and invariant by a screw motion) correspond to the rotationally invariant unduloids and nodoids \cite{HH89, PR99} in $\s^2 \times \R$ and $\h^2 \times \R$ (cf.\ Proposition~\ref{prop:spherical-helicoids}), except for a special case which correspond to a torus in $\s^2\times\R$ or a cylinder in $\h^2\times\R$, both of them invariant by a $1$-parameter group of horizontal isometries.

Section~\ref{sec:cmc-surfaces-products} deals with the construction of a 1-parameter family of complete singly periodic \textsc{cmc} $H$ surfaces in the product spaces $\s^2 \times \R$ ($H>0$) and $\h^2 \times \R$ ($H > 1/2$), coming from minimal surfaces in Berger spheres. The surfaces we obtain are extended by reflection to complete ones which are invariant by discrete $1$-parameter groups of isometries, consisting of rotations in $\s^2\times\R$ or hyperbolic translations in $\h^2\times\R$ (cf.\ Theorem~\ref{thm:unduloids}). In the case of $\h^2\times \R$ the constructed surfaces are proved to be embedded. The case of $\s^2 \times \R$ is more subtle due to the compactness of $\s^2$, and we will be able to show that for $H < 1/2$ there are non-embedded ones. Nonetheless, this family provides \textsc{cmc} $H$ tori for each $H>0$. In both cases, we will give a quite precise description of them. An interesting property of these new examples is that their height is bounded and takes all values in between the height of the rotationally invariant sphere and the height of the rotationally invariant torus in $\s^2\times\R$ (resp.\ cylinder invariant under hyperbolic translations in $\h^2\times\R$). Note that the height of the aforementioned torus (resp.\  cylinder) is a half of that of the corresponding sphere.

Finally, in Section~\ref{sec:cmc-1/2-H2xR} we will construct \textsc{cmc}  $1/2$ bi-multigraphs in $\h^2\times\R$ which have the symmetries of a tessellation of $\h^2$ by regular polygons (cf.\ Theorem~\ref{thm:tiling-surface}), coming from minimal surfaces in the Heisenberg group.  Recall that the value $H=1/2$ is critical in the sense that \textsc{cmc} surfaces for $H > 1/2$ and $H \leq 1/2$ are of different nature (e.g.\ \textsc{cmc} spheres only exist for $H > 1/2$). Besides, we give some applications to the construction of \textsc{cmc} $1/2$ surfaces in $M\times\R$, where $M$ is a compact surface with negative Euler characteristic and constant curvature $-1$. 

We want to mention that \textsc{cmc} $0<H<1/2$ surfaces in $\h^2 \times \R$ can be obtained by the conjugate construction from minimal surfaces in the simplectic group $\Sl$ endowed with an appropriate homogeneous metric, though the analysis of such surfaces is beyond the goal of this paper.

The authors wish to thank professors Joaquín Pérez, Antonio Ros and Francisco Urbano for their comments during the preparation of this paper. We also would like to thank the referee for the thorough revision of the manuscript and the inestimable suggestions to improve it.

\section{Preliminaries on homogeneous $3$-manifolds}\label{sec:preliminaries}

Simply-connected homogeneous Riemannian $3$-manifolds with isometry group of dimension $4$ or $6$, different from the hyperbolic space $\h^3$, form a $2$-parameter family $\E(\kappa,\tau)$, $\kappa,\tau\in\R$ (see~\cite{Daniel07}). Moreover, every $\E(\kappa, \tau)$ admits a fibration over the simply-connected constant curvature $\kappa$ surface whose vertical field $\xi$ is Killing and $\tau$ represents the bundle curvature. In particular, $\E(1, 0)=\s^2\times\R$ and $\E(-1,0)=\h^2 \times \R$.

We will say that a geodesic in $\E(\kappa, \tau)$ is \emph{horizontal} if its tangent vector is orthogonal to $\xi$ and \emph{vertical} if its tangent vector is co-linear with $\xi$. For our purposes, the following property will be essential:
\begin{quote}
\itshape
Given a vertical or horizontal geodesic, there exists a unique involutive isometry of $\E(\kappa,\tau)$ which fixes each point in the geodesic. It will be called a \emph{geodesic reflection} with respect to the geodesic.
\end{quote}

\subsection{The Daniel correspondence}\label{subsec:Daniel-correspondence}

Lawson correspondence~\cite[Section 14]{Lawson70} was generalized by Daniel~\cite[Theorem 5.2]{Daniel07} to the context of $\E(\kappa, \tau)$-spaces. More explicitly, given $\E=\E(\kappa,\tau)$ and $\E^*=\E(\kappa^*,\tau^*)$, such that $\kappa-4\tau^2=\kappa^*-4(\tau^*)^2$ and given $\theta,H,H^*\in\R$ satisfying $H+i\tau=e^{i\theta}(H^*+i\tau^*)$, the following statement holds:
\begin{quote}
\itshape
Let $\phi:\Sigma\rightarrow\E$ an isometric constant mean curvature $H$ immersion of a simply connected surface $\Sigma$. There exists a isometric immersion $\phi^*:\Sigma\rightarrow\E^*$ of \textsc{cmc} $H^*$ such that:
\begin{enumerate}[(a)]
 \item $\nu^*=\nu$
 \item $T^*=e^{\theta J}T$.
 \item $S^*=e^{\theta J}(S-H\cdot\mathrm{id})+H^*\cdot\mathrm{id}$
\end{enumerate}
where $e^{\theta J}$, $\nu=\langle N,\xi\rangle$, $T=\xi-\nu N$ and $S$ are the positive oriented rotation of angle $\theta$ in the tangent plane to $\Sigma$, the angle function, the tangent part of the vertical field and the shape operator for $\phi$, and $N$ is a unit normal vector field to the immersion. The elements $\nu^*$, $T^*$ and $S^*$ are the corresponding ones for $\phi^*$.

The immersion $\phi^*$ is unique up to an ambient isometry in $\E^*$ and is called a \emph{sister immersion} of $\phi$.
\end{quote}

We recall that we are interested in applying the correspondence between a minimal surface in some $\E(\kappa,\tau)$ and a \textsc{cmc} $H$ surface in $M^2(\epsilon)\times\R=\mathbb{E}(\epsilon,0)$, so the parameter $\theta$ must be $\frac{\pi}{2}$ (there is no loss of generality in considering $\theta$ to be positive). Hence, $(\kappa,\tau)$ must be equal to $(4H^2 + \epsilon,H)$, and this leads to the relations
\begin{align}\label{eq:daniel}
\nu^*&=\nu,&T^*&=JT,&S^*&=JS+H\cdot\mathrm{id},
\end{align}
between the sister surfaces. The possible choice of parameters is given by Figure~\ref{fig:table}, as well as, for $H=\tau=0$, the families of associate minimal surfaces in $\s^2\times\R$ or $\h^2\times\R$ (notice that $\theta$ is free in that case).

\begin{figure}
\begin{center}
\renewcommand{\arraystretch}{1.3}
\begin{tabular}{c|ccc}
 &Initial $\E(\kappa,\tau)$& Surfaces in $\h^2\times\R$ & Surfaces in $\s^2\times\R$\\\hline
 $\kappa>0\ $&$\E(4H^2+\epsilon,H)$& \textsc{cmc} $H>1/2$&\textsc{cmc} $H>0$\\
 $\kappa=0\ $&$\E(0,1/2)$& \textsc{cmc} $H=1/2$& -\\
 $\kappa<0\ $&$\E(4H^2-1,H)$& \textsc{cmc} $0<H<1/2$&-
\end{tabular}
\end{center}
\caption{Possible configurations of parameters.}\label{fig:table}
\end{figure} 

We will now precise what the horizontal and vertical geodesics contained in a minimal surface become in the sister surface, for the cases in Figure~\ref{fig:table}. Notice that the choice $\theta=\frac{\pi}{2}$ is instrumental in the proof.

\begin{lemma}\label{lm:simetrias-productos}
Given $\epsilon\in\{-1,1\}$ and $H \geq 0$, let $\phi:\Sigma\looparrowright \mathbb{E}(4H^2 + \epsilon,H)$ be a isometric minimal immersion of a simply connected Riemannian surface $\Sigma$ and suppose $\phi^*:\Sigma\looparrowright M^2(\epsilon)\times\R$ is its sister \textsc{cmc} H immersion. Given a smooth curve $\alpha:[a,b]\rightarrow\Sigma$,
\begin{enumerate}[(a)]
 \item if $\phi(\alpha)$ is a horizontal geodesic, then $\phi^*(\alpha)$ is contained in a vertical plane, which the immersion meets orthogonally.
 \item If $\phi(\alpha)$ is a vertical geodesic, then $\phi^*(\alpha)$ is contained in a horizontal plane, which the immersion meets orthogonally.
\end{enumerate}
\end{lemma}

\begin{proof}
The first part of this lemma was proved by Torralbo (cf.\ \cite[Proposition 3]{Tor10a}) but, for completeness, we include the proof here.
 
We will follow the above notation and consider $\gamma=\phi(\alpha)$ and $\gamma_*=\phi^*(\alpha)=(\beta, h) \subseteq \Sigma \subseteq M^2(\epsilon) \times \R$, where there is no loss of generality in considering $\alpha$ to be parametrized by its arc length. Moreover, it is possible to immerse isometrically  $M^2(\epsilon) \times \R$ in $\subseteq\R^3\times\R$ if $\epsilon=1$ or $\R^3_1 \times \R$ if $\epsilon = -1$ with unit normal along $\gamma$ given by $(\beta, 0)$. 

We will start by proving item (a). We claim that $J\gamma_*'$ is constant, where $J\gamma_*'$ is considered to be a curve in $\R^4$ or $\R^4_1$. 
\[
\begin{split}
\prodesc{(J\gamma_*')'}{\gamma_*'} = -\prodesc{J\gamma'}{\nabla_{\gamma'}\gamma'} &= 0, \qquad \text{(as $\gamma$ is a geodesic)} \\
\prodesc{(J\gamma_*')'}{J\gamma_*'} &= 0, \qquad \text{(as $J\gamma_*'$ has length $1$)} \\
\prodesc{(J\gamma'_*)'}{N^*} = -\prodesc{J\gamma_*'}{\df N^*(\gamma'_*)} &= \prodesc{J\gamma'_*}{S^*\gamma'_*} =   \prodesc{J\gamma'}{-S\gamma' + \tau \gamma'} = \\
=-\prodesc{\gamma'}{S\gamma'} &= 0, \qquad \text{(as $\gamma$ is an asymptote line)} \\
\end{split}
\]
where we take into account the relation~\eqref{eq:daniel}. Thus, the tangent part of $(J\gamma_*')'$ to $M^2(\epsilon)\times\R$ vanishes, so $(J\gamma_*')'$ is proportional to $(\beta, 0)$. On the other hand, as $\gamma$ is an horizontal curve in $\mathbb E(4H^2 + \epsilon,H)$, we know that $\prodesc{\gamma'}{\xi} = \prodesc{\gamma'}{T} = 0$ so $\gamma'$ is proportional to $JT$ and, since it has unit length, we can suppose that, up to a sign, $\gamma' = JT/\sqrt{1-\nu^2}$. Therefore, $T^*=\gamma'_*\sqrt{1-\nu^2}$ and $(0, 1) =T^*+\nu N^*= \sqrt{1 - \nu^2}\gamma_*' + \nu N^*$.

This last relation implies that $\prodesc{J\gamma_*'}{(0, 1)} = 0$. Hence,
\[
\prodesc{(J\gamma_*')'}{(\beta, 0)} = -\prodesc{J\gamma_*'}{(\beta', 0)} = h'\prodesc{J\gamma_*'}{(0, 1)} = 0,
\]
where we have used that $0 = \prodesc{J\gamma_*'}{\gamma_*'} = \prodesc{J\gamma_*'}{(\beta', 0)} + \prodesc{J\gamma_*'}{h'(0,1)}$, and the claim is proved. 

In fact, we have proved that $J\gamma_*' = (v, 0) \in \R^3 \times \R$ for some fixed $v\in TM^2(\epsilon)\subset\R^3$. Taking this into account, 
\[\prodesc{\gamma_*}{(v,0)}'= \prodesc{\gamma_*'}{(v,0)} = \prodesc{\gamma_*'}{J\gamma_*'} = 0\]
which implies that $\prodesc{\gamma_*}{(v,0)}$ is constant, but $
\prodesc{\gamma_*}{(v,0)} = \prodesc{\beta}{v} = 0$ as $\beta$ is normal to $M^2(\epsilon)$ and $v$ is tangent. 

All this information says that $\gamma_*$ lies in the vertical plane 
$P= \{(p, t) \in M^2(\epsilon) \times \R:\, \prodesc{p}{v} = 0\}$. Moreover, the immersion $\phi_*$ is orthogonal to $P$ since the tangent plane along $\gamma_*$ is spanned by $\{\gamma_*', J\gamma_*' = (v,0)\}$.

Let us now prove item (b). Observe first that, if $\gamma$ is a vertical geodesic, then $\nu=0$ along it, so $\xi^*=T^*+\nu N^*=T^*$ along $\gamma_*$. Thus, 
\[\prodesc{\gamma_*'}{\xi^*}=\prodesc{\gamma_*'}{T^*}=\prodesc{\gamma'}{JT}=0.\]
The last equality follows from the fact that $\gamma'$ is vertical whereas $JT$ is horizontal ($T$ is vertical along $\gamma$ since $\gamma$ is a vertical curve contained in the surface). Finally, notice that $0=\prodesc{\gamma_*'}{\xi^*}=\prodesc{(\beta',h')}{(0,1)}=h'$, so $h$ is constant along $\gamma_*$, i.e.\  this curve is contained in a horizontal slice. As $\nu=0$ along $\gamma_*$ the surface meets that slice orthogonally.
\end{proof}

\subsection{Smooth extension of surfaces bordered by geodesics}\label{subsec:smooth-extension}

Let $\Sigma$ be a minimal surface immersed in $\E(\kappa, \tau)$ and $\gamma$ a vertical or horizontal geodesic of $\E(\kappa, \tau)$ contained in $\partial \Sigma$. Then, it is possible to extend $\Sigma$ by geodesic reflection around $\gamma$ (we recall that every geodesic reflection, whenever the geodesic is horizontal or vertical, is an isometry of $\E(\kappa, \tau)$). This extension is smooth in view of~\cite[vol.~II, Theorem 4]{DHKW}.

Moreover, taking into account Lemma~\ref{lm:simetrias-productos} and the data $(\nu, T, S)$ of a minimal surface invariant with respect to either a vertical or horizontal geodesic reflection, it is easy to prove the following result that establishes the behaviour of this symmetries with respect to the Daniel correspondence.

\begin{lemma}\label{lm:simetrias-correspondencia-daniel}
Let $\phi: \Sigma \looparrowright \E(4H^2 + \epsilon, H)$ be a minimal immersion of a simply connected surface and $\phi^*:\Sigma \looparrowright M^2(\epsilon) \times \R$ its sister immersion. Then:
\begin{enumerate}[(i)]
	\item If $\phi(\Sigma)$ is invariant by a horizontal (resp.\ vertical) geodesic reflection then $\phi^*(\Sigma)$ is invariant by a reflection over a vertical (resp.\ horizontal) plane.
	\item The axis of reflection in the original surface corresponds to the curve where the sister immersion meets de plane of reflection.
\end{enumerate}
\end{lemma}

Another interesting situation that will often appear is when $\partial\Sigma$ contains two different vertical or horizontal geodesics meeting at some point $p\in\partial\Sigma$. Then, the surface can be extended by reflection over both geodesics producing, in each step, a new vertical or horizontal geodesic passing through $p$. If the angle between the two geodesics is $\frac{\pi}{k}$ for some $k \in \N$, then a surface is produced after $2k$ reflections, which is smooth in every point except possibly at $p$. Nevertheless, if such a surface is locally embedded around the point $p$ then, thanks to the removable singularity result~\cite[Proposition 1]{CS85}, it will be also smooth at $p$.

\section{Models for homogeneous spaces}\label{sec:homogeneous-spaces}
For our purposes we will restrict ourselves to the construction of minimal surfaces in $\E(\kappa,\tau)$ spaces where $\tau\neq0$. In what follows we will only consider the cases where $\kappa\geq 0$ (i.e.\ either the Berger spheres or the Heisenberg group) so \textsc{cmc} $H > 0$ in $\s^2 \times \R$ and $H \geq 1/2$ in $\h^2 \times \R$ will be produced. In this section we will introduce briefly the aforementioned homogeneous spaces, focusing on the properties needed in the paper.

\subsection{The Berger spheres}
A Berger sphere is a $3$-sphere $\s^3 = \{(z, w) \in \C^2:\, \abs{z}^2 + \abs{w}^2 = 1\}$ endowed with the metric
\[
 g(X, Y) = \frac{4}{\kappa}\left[\prodesc{X}{Y} + \left(\frac{4\tau^2}{\kappa} - 1\right)\prodesc{X}{V}\prodesc{Y}{V} \right],
\]
where $\prodesc{\,}{\,}$ stands for the usual metric on the sphere, $V_{(z, w)} = J(z, w) = (iz, iw)$, for each $(z, w) \in \s^3$ and $\kappa$, $\tau$ are real numbers with $\kappa > 0$ and $\tau \neq 0$. From now on, we will denote the Berger sphere $(\s^3,g)$ as $\s^3_b(\kappa, \tau)$. We note that if $\kappa = 4\tau^2$ then $\s^3_b(\kappa, \tau)$ is, up to homotheties, the round sphere. The Berger spheres are examples of $\E(\kappa, \tau)$ for $\kappa > 0$ and $\tau \neq 0$ (cf.\ \cite{Tor10b} for a detailed description).

The Hopf fibration $\Pi: \s^3_b(\kappa, \tau) \rightarrow \s^2(\kappa)$, where $\s^2(\kappa)$ stands for the $2$-sphere of radius $1/\sqrt{\kappa}$, given by
\[
\Pi(z, w) = \frac{2}{\sqrt{\kappa}}\left(z\bar{w}, \frac{1}{2}(\abs{z}^2 - \abs{w}^2) \right),
\]  
is a Riemannian submersion whose fibers are geodesics. The vertical unit Killing field is given by $\xi = \frac{\kappa}{4\tau}V$.  It is easy to check that both the horizontal and vertical geodesic are great circles. It is interesting to remark that the length of every vertical geodesic is $8\tau\pi/\kappa$, whereas the length of every horizontal geodesic is $4\pi/\sqrt{\kappa}$.

\subsection{The product spaces}\label{subsec:product-spaces}
As it has been pointed out before, the only homogeneous spaces with isometry group of dimension four and zero bundle curvature are the Riemannian products $M^2(\kappa)\times\R=\E(\kappa,0)$, where $M^2(\kappa)$ stands for the simply connected surface with constant curvature $\kappa$. The Riemannian submersion coincides with the natural projection $\Pi:M^2(\kappa)\times\R\to M^2(\kappa)$.

Totally geodesic surfaces of $M^2(\kappa) \times \R$ are either \emph{vertical planes}, i.e.\ the product of a geodesic of $M^2(\kappa)$ with the real line (they are topological cylinders if $\kappa>0$ and planes if $\kappa<0$) or \emph{horizontal planes} (also called \emph{slices}), i.e.\ $M^2(\kappa) \times \{t_0\}$, $t_0 \in \R$. It is well known that the reflection over a horizontal or vertical plane is an ambient isometry.

If a \textsc{cmc} surface $\Sigma$ meets a horizontal or vertical plane orthogonally, we can smoothly extend this surface by reflecting over this plane. This is a consequence of the continuation result of Aronszajn~\cite{Aron57} for elliptic \textsc{pde}'s joint with the fact that the reflection over horizontal and vertical planes are ambient isometries.

\subsection{The Heisenberg group} \label{subsec:heisenberg} The Heisenberg group $\Nil=\E(0,\frac{1}{2})$ is a Lie group whose Lie algebra consists of the upper-triangular nilpotent $3\times 3$ real matrices. It can be modeled by $\R^3$, endowed with the metric
\[
\mathrm{d}s^2 = \mathrm{d}x^2 + \mathrm{d}y^2 + \bigl(\tfrac{1}{2}(y\mathrm{d}x - x\mathrm{d}y)  + \mathrm{d}z\bigr)^2,
\]
where $(x, y,z)$ are the usual coordinates of $\R^3$. The projection $\Pi:\Nil \rightarrow \R^2$ given by $\Pi(x, y, z) = (x, y)$ is a Riemannian submersion and $\partial_z$ is a unit vertical Killing vector field.

All vertical and horizontal geodesics in $\Nil$ are Euclidean straight lines, not necessarily linearly parametrized. Moreover, every non-vertical Euclidean plane is minimal and any two of them are congruent by an ambient isometry. Vertical Euclidean planes are also minimal in $\Nil$.

\section{Spherical helicoids and their correspondent sister surfaces}\label{sec:spherical-helicoids}

In this section we are going to illustrate the construction method. For that purpose, we focus on the Berger sphere case, i.e.\ we are going to work in $\s^3_b(4H^2 + \epsilon, H)$ and analyse what the correspondent to the so-called \emph{spherical helicoids} are. The latter form a $1$-parameter family of well known minimal immersions in the round $3$-sphere, introduced by Lawson in~\cite{Lawson70} and given by 
\[
\begin{split}
\Phi_c: \R^2 &\rightarrow \s^3\subset\C^2 \\
(x, y) &\mapsto \bigl( \cos(x) e^{icy}, \sin(x) e^{iy} \bigr).
\end{split}
\]

All these immersions are minimal in $\s^3_b(\kappa, \tau)$, for any $\kappa$ and $\tau$. In fact, they are the only immersions in the $3$-sphere that are minimal with respect to all the Berger metrics (cf.\ \cite[Proposition 1]{Tor10a}).

\begin{remark}\label{rmk:spherical-helicoids}~
\begin{enumerate}[(1)]
	\item We can restrict the parameter $c$ to the interval $[-1,1]$ since the surfaces $\Phi_{1/c}$ and $\Phi_c$ are congruent up to a reparametrization, i.e. $(L\circ \Phi_{1/c})(\frac{\pi}{2}-x, cy) = \Phi_c(x, y)$, where $L(z, w) = (w, z)$.
	\item $\Phi_0:]0,\pi[\times [0, 2\pi] \rightarrow \s^3$ is the minimal sphere (except two points), which is embedded. On the other hand, if $c\in\mathbb Q$, then the immersion $\Phi_c$ is induced to a torus. Moreover,  $\Phi_1$ is the Clifford torus, and $\Phi_0$ and $\Phi_1$ are the only embedded spherical helicoids since $\Phi_c(\frac{\pi}{2}, \frac{2\pi}{c}) = \Phi_c(\frac{\pi}{2}, \frac{2\pi}{c}(1-c))$ and $\frac{2\pi}{c} = \frac{2\pi}{c}(1-c) \pmod{2\pi}$ if and only if $c \in \{0, 1\}$. Observe that a Clifford torus is nothing but the lift by the Hopf projection of a geodesic in $\s^2(\kappa)$. Given $p \in \s^3_b(\kappa, \tau)$ and a horizontal vector $u$ at $p$ there exists a unique Clifford torus passing through $p$ with tangent plane at $p$ orthogonal to $u$.
	\item For every $c$, the surface $\Phi_c$ is invariant by the $1$-parameter group of isometries $t \rightarrow \left(\begin{smallmatrix} e^{ict} & 0 \\ 0 & e^{it} \end{smallmatrix} \right)$.
\end{enumerate}
\end{remark}

We will now focus on the case $c \neq -1$, because the sister surface of the spherical helicoid $\Phi_{-1}$ is of different nature and will be treated in Section~\ref{sec:sister-surfaces}. Let us consider the polygon $\Lambda_c$, $c \neq -1$, consisting of the curves:
\[
\begin{aligned}
h_1(t) &= \bigl(\cos(t), \sin(t)\bigr) = \Phi_c(t, 0), & t &\in \left[0, \tfrac{\pi}{2}\right], \\
h_2(t) &= \bigl(\cos(t)e^{\frac{i\pi c}{2(1+c)}}, \sin(t)e^{\frac{i\pi}{2(1+c)}}\bigr) = \Phi_c\left(t, \tfrac{\pi}{2(1+c)}\right), & t &\in \left[0, \tfrac{\pi}{2}\right], \\
v_1(t) &= (e^{ict}, 0) = \Phi_c(0, t), & t &\in \left[0, \tfrac{\pi}{2(1+c)}\right], \\
v_2(t) &= (0, e^{it}) = \Phi_c\left(\tfrac{\pi}{2}, t\right), & t &\in \left[0, \tfrac{\pi}{2(1+c)}\right]. \\
\end{aligned}
\]
(see Figure~\ref{fig:spherical-helicoids}). It is easy to check that the curve $h_\theta(t) = \Phi_c(t, \theta)$, $\theta \in [0, \tfrac{\pi}{2(1+c)}]$, is a horizontal geodesic and $v_1$, $v_2$ are vertical ones, for all $c$. Moreover, we can recover the whole surface $\Phi_c(\R^2)$ by geodesic reflection of the piece $\Phi_c([0, \tfrac{\pi}{2}]\times[0, \tfrac{\pi}{2(1+c)}])$ across the edges of $\Lambda_c$.

Consider now the sister immersion $\Phi^*_c:[0,\frac{\pi}{2}]\times[0,\tfrac{\pi}{2(1+c)}] \rightarrow M^2(\epsilon) \times \R$ and denote by $h_\theta^*$ and $v_j^*$, $j = 1, 2$, the corresponding curves. In view of Lemma~\ref{lm:simetrias-productos}, $h_\theta^*$ are contained in a vertical plane of symmetry, $P_\theta$, while $v_j^*$ is contained in a slice $M^2(\epsilon) \times \{p_j\}$, $j = 1, 2$. To understand the behavior of these curves, one can compute their curvature as curves in the vertical or horizontal plane they lie in. Let us observe firstly that, since the sister surface intersects the slice and the vertical plane where the curves $h_j^*$ and $v_j^*$ meet orthogonally, the curvatures of these two curves (supposed to be parametrized by arc length) are given by
\[
\begin{split}
k_{v_j^*}^{M^2(\epsilon) \times\{p_j\}} &= \prodesc{S^* (v_j^*)'}{(v_j^*)'} = H - \prodesc{S v_j'}{J v_j'},\\
k_{h_\theta^*}^{P_\theta} &= \prodesc{S^* (h_\theta^*)'}{(h_\theta^*)'} = H - \prodesc{S h_\theta'}{J h_\theta'},
\end{split}
\]
where $S$ and $S^*$ denote the shape operators of $\Phi_c$ and $\Phi^*_c$, respectively. The second identity follows from the last equation in~\eqref{eq:daniel}.

Finally, as we know explicitly the shape operator of $\Phi_c$,  straightforward computations show that $v_j^*$ are constant curvature curves in $M^2(\epsilon) \times \{p_j\}$, $j = 1, 2$. On the other hand, all the curves $h_\theta^*$ have the same curvature since the immersion $\Phi_c$ is invariant by a $1$-parameter group of isometries that transform each $h_{\theta_1}$ into another $h_{\theta_2}$. Hence, every point of $v_j^*$ is contained in a vertical plane of symmetry so the sister surface must be rotationally invariant. 
 
 \begin{proposition}\label{prop:spherical-helicoids}
 	The sister surface of the spherical helicoid $\Phi_c$ is a rotationally invariant surface. More precisely,
 	\begin{enumerate}[(i)]
 		\item The sister surface of the minimal sphere $\Phi_0$ is the constant mean curvature $H$ sphere.
 		\item The sister surface of the Clifford torus $\Phi_1$ is the vertical cylinder, i.e.,\ the product of a constant curvature $2H$ curve of $M^2(\epsilon)$ with the real line.
 		\item The sister surface of $\Phi_c$ for $0 < c < 1$ is an unduloid \cite[Lemma 1.3]{PR99}.
 		\item The sister surface of $\Phi_c$ for $-1 < c < 0$ is a nodoid \cite[Lemma 1.3]{PR99}.
 	\end{enumerate}
 \end{proposition}
 
\begin{figure}[htbp]
\centering
\includegraphics{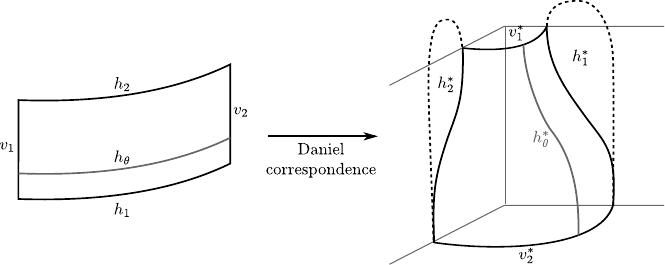}
\caption{Polygon $\Lambda_c$ ($c \neq -1$) in the Berger sphere (left) and its sister contour in $M^2(\epsilon) \times \R$ (right) for $c > 0$ (solid line) and $c < 0$ (dashed line)}\label{fig:spherical-helicoids}
\end{figure}

\begin{proof}
On the one hand, the previous argument shows that the sister surface of $\Phi_c$ must be a rotationally invariant \textsc{cmc} $H$ surface in $M^2(\epsilon) \times \R$. On the other hand, the first assertion is trivial and the second one is easy because the Clifford torus has vanishing constant angle, which remains invariant under the Daniel correspondence. Finally, (iii) and (iv) is a consequence of a deep analysis of the curvature of the curves $h_\theta$, which will be omitted since it is long and straightforward.
\end{proof}

\begin{remark}
The previous argument shows that every minimal surface which is ruled by horizontal geodesics becomes, via the Daniel correspondence, a \textsc{cmc} surface invariant by a $1$-parameter group of isometries.

In the round sphere case (via the Lawson correspondence), the corresponding surfaces to the spherical helicoids are the Delaunay \textsc{cmc} rotationally invariant examples in $\R^3$~\cite[Theorem 2.1]{GB93}.
\end{remark}

\section{Constant mean curvature surfaces in $\s^2 \times \R$ and $\h^2 \times \R$}\label{sec:cmc-surfaces-products}

In this section we will construct, for each $H > 0$ (resp.\ $H>1/2$), a $1$-parameter family of constant mean curvature $H$ surfaces in $\s^2 \times \R$ (resp.\ $\h^2 \times \R$). We will first build a 1-parameter family of minimal surfaces $\Sigma_\lambda$ in the Berger sphere $\s^3_b(4H^2+\epsilon, H)$, $\epsilon \in \{-1,1\}$, by solving the Plateau problem over an appropriate geodesic polygon $\Gamma_\lambda$ (see figure~\ref{fig:gamma-lambda}). The desired surfaces will be the corresponding \textsc{cmc} $H$ surfaces.

\subsection{Geodesic polygons}\label{subsec:geodesic-polygons} Let us consider $\lambda\in[0,\frac{\pi}{2}]$ a real parameter and define the geodesic polygon $\Gamma_\lambda$, explicitly parametrized as
\[
\begin{aligned}
h_0(t) &= \frac{1}{\sqrt{2}}(e^{it}, e^{-it}), & t &\in \left[0, \tfrac{\pi}{2}\right] ,\\
h_1(t) &= (\cos t, \sin t),  &t &\in \left[\tfrac{\lambda}{2}, \tfrac{\pi}{4}\right],\\
h_2(t) &= (i\cos t, i\sin t), & t &\in \left[-\tfrac{\pi}{4}, \tfrac{\lambda}{2}\right],\\
v(t) &= \left( e^{it}\cos \tfrac{\lambda}{2} , e^{it} \sin \tfrac{\lambda}{2}\right), &t&\in \left[0, \tfrac{\pi}{2}\right].
\end{aligned}
\]
Notice that $h_0$, $h_1$ and $h_2$ are horizontal geodesics which project by $\Pi$ on two orthogonal great circles ($h_1$ and $h_2$ in the same one), whilst $v$ is a vertical one (none of them are arc-length parametrized). In fact, the dependence on $\lambda$ lies in where we choose the point to split the geodesic and do the horizontal lift (see Figure~\ref{fig:gamma-lambda}). Observe that two consecutive curves of $\Gamma_\lambda$ meet at a $\frac{\pi}{2}$ angle.

\begin{figure}[htbp]
\centering
\includegraphics{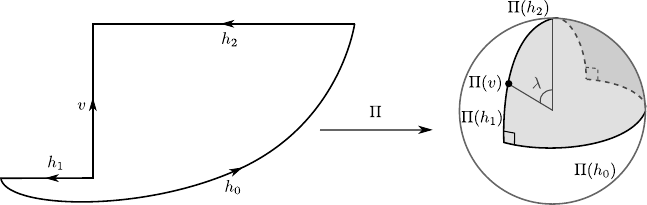}
\caption{Polygon $\Gamma_\lambda$ (left) and  its Hopf projection $\Pi(\Gamma_\lambda)$ (right). The parameter $\lambda$ represents the marked angle.}\label{fig:gamma-lambda}
\end{figure}

In the sequel, let us consider $W=\Pi^{-1}(\hat W)$, where $\hat{W}$ is the convex compact domain in $\s^2(4H^2 + \epsilon)$ bounded by $\Pi(\Gamma_\lambda)$, which is nothing but a quarter of the sphere. Thus $W$ is a solid torus whose boundary is made out of two pieces of Clifford tori (cf.\ Remark \ref{rmk:spherical-helicoids}) which meet at a $\frac{\pi}{2}$ angle.

\begin{proposition}\label{prop:existencia-Plateau-gamma-lambda}
There exists a unique minimal surface $\Sigma_\lambda\subset W$ with border $\Gamma_\lambda$. Moreover, the interior of $\Sigma_\lambda$ is a graph over $\hat W$ and can be extended smoothly across its boundary.
\end{proposition}

\begin{proof}
We know that $W$ is a mean convex body and it is clear $\Gamma_\lambda$ is nullhomotopic in $W$, so the existence follows from the results by Meeks and Yau~\cite{MY82}.  Hence, there exists a minimal surface $\Sigma_\lambda$ with border $\Gamma_\lambda$ which is $C^2$ in the interior and $C^0$ in the boundary. 

Notice that, working in the Riemannian universal cover of $W$, the maximum principle for minimal graphs in Killing submersions given by Pinheiro in \cite[Theorem 2.1]{Pinheiro} (see also~\cite[Corollary 3.9, p.\ 100]{ManPhD}) ensures that there exists at most one minimal graph in $W$ with boundary $\Gamma_\lambda$. Therefore, to finish the proof it suffices to show that $\Sigma_\lambda$ is a graph, i.e.\ its angle function $\nu$ never vanishes.

Let us consider $\Gamma$, a small deformation of $\Gamma_\lambda$ in a neighbourhood of the vertical component $v$ such that $\Pi\lvert_{\Gamma}$ is injective and $\Gamma$ lies \emph{above} $\Gamma_\lambda$. Then, by the existence result and a classic application of the maximum principle, there exists a minimal graph $\Sigma$ with border $\Gamma$ and it lies above $\Sigma_\lambda$. We can consider a decreasing sequence $\{\Gamma_n\}_n$ of such deformations (i.e.\ $\Gamma_n$ lies above $\Gamma_{n+1}$ for every $n$) that converge to $\Gamma_\lambda$. The corresponding solutions $\Sigma_n$ to the Plateau problem with border $\Gamma_n$ will be graphs and will converge to a minimal surface $\Sigma^+$ with border $\Gamma_\lambda$ (each of the surfaces $\Sigma_n$ is stable and their geometries are uniformly bounded so standard converge arguments can be applied, note also that the sequence is monotonic). Hence, $\Sigma^+$ is above $\Sigma_\lambda$ and it is a graph. This last property is a consequence of being limit of graphs so its angle function does not change sign: as the angle function is a Jacobi function and $\Sigma^+$ is stable, then is either identically zero or never vanishes. The first case is obviously not possible.

\begin{figure}[htbp]
\centering
\includegraphics{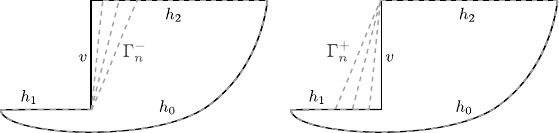}
\caption{Deformations $\Gamma^-_n$ (left) and $\Gamma^+_n$ (right) of the original polygon $\Gamma_\lambda$ whose associated Plateau solutions converge to the minimal graphs $\Sigma^+$ and $\Sigma^-$.}
\end{figure}

Likewise, we can deform the curve $\Gamma_\lambda$ so the new one will be below and construct a minimal graph $\Sigma^-$ with border $\Gamma_\lambda$ below $\Sigma_\lambda$. Finally, Pinheiro's argument shows that $\Sigma^+ = \Sigma^-$. Thus, $\Sigma_\lambda = \Sigma^+$ and, in particular, $\Sigma_\lambda$ is a graph. The smooth extension property follows from Section~\ref{subsec:smooth-extension}.
\end{proof}

\begin{remark}\label{rmk:casos-particulares-Sigma-lambda}~
\begin{enumerate}[(1)]
\item If $\lambda = 0$ or $\lambda = \tfrac{\pi}{2}$ it is easy to check, due to the uniqueness of $\Sigma_\lambda$, that $\Sigma_0$ is a piece of the spherical helicoid $\Phi_{-1}$ (cf.\ Proposition~\ref{prop:existencia-Plateau-gamma-lambda}) and $\Sigma_{\pi/2}$ is a piece of the sphere $\{(z,w) \in \s^3:\, \pRe(z-w) = 0\}$.

\item The deformation technique used in the proof can be applied to show the existence and uniqueness of graphical solutions of the Plateau problem in Killing submersions for a wide family of contours, known as Nitsche contours (see also \cite[Theorem 3.11, p.\ 101]{ManPhD}).
\end{enumerate}
\end{remark}

Now, we will focus on the dependence of the family $\Sigma_\lambda$ on $\lambda$. First, notice that $h_0$ does not depend on $\lambda$ and both $h_1$ and $h_2$ lie in horizontal geodesics which differ on a vertical translation. Thus, working in the universal cover of $W$ and given $0\leq\lambda_1<\lambda_2\leq\frac{\pi}{2}$, the maximum principle applied to $\Sigma_{\lambda_1}$ and $\Sigma_{\lambda_2}$ ensures that they do not intersect except in the common boundary. This fact proves that the family $\{\Sigma_\lambda:0\leq\lambda\leq\frac\pi 2\}$ is vertically ordered with respect to the parameter $\lambda$.

We claim it defines a foliation of the domain $U\subset W$ bounded  by the sphere $\Sigma_{\pi/2}$, the spherical helicoid $\Sigma_{0}$ and the Clifford torus $\Pi^{-1}(\Pi(h_2))$. To prove the claim, it suffices to check that $U=\cup_{\lambda=0}^{\pi/2}\Sigma_\lambda$, so we will prove that there exists no $p_0\in U$ such that it is not contained in any $\Sigma_\lambda$. If this situation occurred then, as the family is vertically ordered, we could define $\lambda_0$ such that $\Sigma_\lambda$ lied in one side of $p_0$ for $\lambda>\lambda_0$ and in the other side for $\lambda<\lambda_0$. Since the involved surfaces are stable, we can take limits for $\lambda\nearrow\lambda_0$ and $\lambda\searrow\lambda_0$. The limit surfaces are minimal and have the same boundary $\Gamma_{\lambda_0}$ so, by uniqueness (cf.\ Proposition~\ref{prop:existencia-Plateau-gamma-lambda}), $p_0$ must lie in $\Sigma_{\lambda_0}$.

Take a point $p \in h_0(]0, \frac{\pi}{2}[) \cup h_1(]0, \frac{\pi}{4}[)$. If $p \in\Gamma_{\lambda_0}$ for some $\lambda_0$, then there exists $\varepsilon > 0$ such that $p \in \Gamma_\lambda$ for all $\lambda \in [0, \lambda_0 + \varepsilon[$. On the other hand, let
$p \in h_2(]-\frac{\pi}{4}, \frac{\pi}{4}[)$. If $p \in \Gamma_{\lambda_0}$, for some $\lambda_0$, then there exists $\varepsilon > 0$ such that $p \in \Gamma_\lambda$ for all $\lambda \in\ ]\lambda_0 - \varepsilon, \frac{\pi}{2}]$. Hence, it makes sense study the function $\lambda \mapsto \nu_\lambda(p)$ for $\lambda$ in the appropriate interval, which is the purpose of the following lemma.

We choose a unit normal vector field $N$ to $\Sigma_\lambda$ so the angle function $\nu_\lambda  = \prodesc{N}{\xi}$ is negative in $\Sigma_\lambda$.

\begin{lemma}\label{lm:crecimiento-angulo-frontera}
The angle function $\nu_\lambda$ of the suface $\Sigma_\lambda$ satisfies:
\begin{enumerate}[(i)]
	\item $\nu_\lambda(p) = 0$ for $p \in \Gamma_\lambda$ if and only if $p \in v([0,\frac{\pi}{2}])$.

	\item $\nu_\lambda$ does not take the value $-1$ in  $\Sigma_\lambda$, \label{lm:crecimiento-angulo-frontera:item:cojonudo} and it only takes that value in $\Gamma_\lambda$ at $h_0(0)$ and $h_0(\tfrac{\pi}{2})$ for $0 < \lambda < \tfrac{\pi}{2}$.
	\item If $p\in h_0(]0, \frac{\pi}{2}[)\cup h_1(]0, \frac{\pi}{4}[)$,  the function $\lambda\mapsto\nu_\lambda(p)$ is continuous and strictly increasing.
	\item If $p\in h_2(]-\frac{\pi}{4}, \frac{\pi}{4}[)$, the function $\lambda\mapsto\nu_\lambda(p)$ is continuous and strictly decreasing.
\end{enumerate}

\end{lemma}

\begin{proof}
To prove \emph{(i)}, if $\nu_\lambda(p)=0$ for some $p\in\Gamma_\lambda$ not lying in $v$, then $p\in h_i$ for some $i\in\{0,1,2\}$. Moreover, $p$ must lie in the interior of the curve $h_i$ since, at its vertices, the angle function either has value$-1$ or $p$ also lies in $v$. Thus, the piece of Clifford torus given by $\Pi^{-1}(\Pi(h_i))$ is tangent to $\Sigma_\lambda$ at $p$ and we get a contradiction to the boundary maximum principle.

Next we prove \emph{(ii)} by contradiction. Let us suppose that there is an interior point $p\in\Sigma_\lambda$ such that $\nu_\lambda(p)=-1$, and consider the surface $\Lambda$ consisting of the horizontal geodesics passing through $p$, which is in fact a minimal sphere, tangent to $\Sigma_\lambda$ at $p$. Then, the intersection $\Lambda\cap\Sigma_\lambda$ forms a system of differentiable curves which meet transversely at some points (where both surfaces are tangent). Thus, at the point $p$ at least two of them meet, but they cannot enclose a compact region (due to the maximum principle) so they necessarily die in $\partial\Sigma_\lambda$ (notice that in the Berger spheres, the umbrella $\Lambda$ is a sphere). If we prove that $\Lambda$ intersects $\partial\Sigma_\lambda$ in two points at most, the contradiction will be clear. On the one hand, showing that $\Lambda$ cannot intersect the vertical boundary twice is an explicit computation and, on the other hand, if $\Lambda$ intersected the horizontal boundary twice, as $\Lambda$ is made out of horizontal geodesics starting at $p$, we would find a closed horizontal geodesic polygon, which projects one-to-one by the Hopf projection, and this is impossible.

Items \emph{(iii), (iv)} hold since the family $\{\Sigma_\lambda: 0 \leq \lambda \leq \frac{\pi}{2}\}$ is vertically ordered and foliates the domain $U$, as we showed before. Finally the second part of \emph{(ii)} follows from \emph{(iii)}, \emph{(iv)} and the well-known behavior of the angle function on $\Sigma_0$ and $\Sigma_{\pi/2}$ along their boundary. To be more precise, let us distinguish two cases:
\begin{itemize}
 \item If $p\in h_2(]-\frac{\pi}{4}, \frac{\pi}{4}[)$, we compare with $\Sigma_{\pi/2}$. Since $\nu_\lambda$ is strictly decreasing along this border, we get that $\nu_\lambda(p)>\nu_{\pi/2}(p) > -1$.
 
 \item If $p\in h_0(]0, \frac{\pi}{2}[) \cup h_1(]0, \frac{\pi}{4}[)$, we compare with $\Sigma_0$. As $\nu_\lambda$ is strictly increasing along this border, we get that $\nu_\lambda(p)>\nu_0(p)\geq -1$.\qedhere
\end{itemize}
\end{proof}

\subsection{Properties of the conjugate surface}\label{sec:sister-surfaces}

\noindent In this section, we will consider the \textsc{cmc} $H$  sister surfaces in $\s^2 \times \R$ (for $H>0$) or $\h^2 \times \R$ (for $H>1/2$) corresponding to $\Sigma_\lambda$, which will be denoted by  $\Sigma_\lambda^*$. 

First of all, recalling Section~\ref{subsec:smooth-extension}, we know that $\Sigma_\lambda$ can be extended to a simply connected minimal surface in such a way $\overline\Sigma_\lambda$ lies in its interior so the Daniel correspondence may be applied to the extended surface and it provides an isometry between $\overline\Sigma_\lambda$ and $\overline\Sigma_\lambda^*$ by restriction. This property guarantees that the lengths of the components of the boundary and the angles they make are preserved. We will denote by $h_0^*$, $h^*_1$, $h_2^*$ and $v^*$ the corresponding curves of the boundary of $\Sigma_\lambda^*$. In view of Lemma~\ref{lm:simetrias-productos}, the curves $h_j^*$, $j\in\{0,1,2\}$, are contained in vertical planes $P_j$ and the surface meets these planes orthogonally. Hence, the angle that $P_i$ and $P_j$ make is the same as that the curves $h_i^*$ and $h_j^*$ make, which is in turn the same that $h_i$ and $h_j$ make, so $P_0$ is orthogonal to $P_1$ and $P_1$ is orthogonal to $P_2$. On the other hand, the curve $v^*$ is contained in a horizontal slice, that will be supposed to be $M^2(\epsilon)\times\{0\}$ after a vertical translation. Hence, we can extend the piece $\Sigma^*_\lambda$ to a complete \textsc{cmc} H surface in $M^2(\epsilon) \times \R$. 

We will first analyse the extremal cases $\lambda=0$ and $\lambda=\frac{\pi}{2}$. The surface $\Sigma_{\pi/2}$ is a part of a minimal sphere in $\s^3_b(4H^2 + \epsilon,H)$ (cf.\ Remark~\ref{rmk:casos-particulares-Sigma-lambda}) and so $\Sigma_{\pi/2}^*$ must be a piece of the \textsc{cmc} $H$ rotationally sphere in $M^2(\epsilon) \times \R$. On the other hand, the surface $\Sigma_{0}$ satisfies $\nu=1$ along the horizontal geodesic $h_0$. Thus, $h_0^*$ has constant height and it is contained in a vertical plane so it must be a horizontal geodesic in $M^2(\epsilon)\times\R$. Moreover, $\Sigma_0$ is foliated by horizontal geodesics (i.e.\ $t\mapsto (e^{i\theta}\cos t, e^{-i\theta}\sin t)$, $\theta \in [0,\pi/2]$), orthogonal to the boundary curve $v$. Thus, $\Sigma^*_0$ is foliated by curves $\gamma_t$ satisfying that:
\begin{itemize}
 \item $\gamma_t$ connects the point $h_0^*(t)$ to a point in $v^*$.
 \item $\gamma_t$ is contained in a vertical plane, orthogonal to $h_0^*$ and $v^*$.
 \item The curves $\Pi\circ\gamma_t$ satisfy $\|(\Pi\circ\gamma_t)'\|=-\nu$, so all of them have the same length.
\end{itemize} 
Thus, the curve $v^*$, which is contained in a horizontal slice, must be a curve equidistant to $\Pi\circ h_0^*$ in $M^2(\epsilon)$. In particular, $v^*$ has constant geodesic curvature in $M^2(\epsilon)$. From the parametrization of $\Sigma_0$ given in Section~\ref{sec:spherical-helicoids}, it is easy to show that such geodesic curvature is given by $\kappa_g=\frac{\epsilon}{2H}.$

In view of \cite[Thm.\ 4.2]{Man10}, the surface $\Sigma_0^*$ must be a part of a rotationally invariant torus if $\epsilon=1$ or a part of a horizontal cylinder, invariant under hyperbolic translations, if  $\epsilon=-1$.

\begin{theorem}\label{thm:unduloids}
Given $\epsilon\in\{-1,1\}$ and $H>0$ with $4H^2+\epsilon>0$, there exists a $1$-parameter family $\{S_\lambda(H):\, \lambda \in [0, \pi/2]\}$ of complete constant mean curvature $H$ surfaces in $M^2(\epsilon) \times \R$. All of them have a horizontal plane of symmetry and they are invariant under a discrete $1$-parameter group of isometries acting trivially on the factor $\R$ (consisting of rotations if $\epsilon = 1$, or hyperbolic translations if $\epsilon = -1$). Furthermore:
	\begin{enumerate}[(i)]
		\item $S_0(H)$ is the \textsc{cmc} $H$ rotationally invariant torus (resp. cylinder) in $\s^2\times \R$ (resp. $\h^2 \times \R$) and $S_{\pi/2}(H)$ is the \textsc{cmc} $H$ sphere.
		\item If $\epsilon=-1$, all the surfaces in the family are embedded.
		\item If $\epsilon = 1$ then for each $H < 1/2$ there exists $\lambda^*$ such that for all $\lambda \geq \lambda^*$ the surfaces $S_\lambda(H)$ are not embedded.
	\end{enumerate}
Moreover, the maximum height of the surface varies continuously between the maximum height of the upper half of the horizontal cylinder and the maximum height of the hemisphere.
\end{theorem}

\begin{remark}\label{rmk:unduloids}~
\begin{enumerate}[(1)]
	\item In the case of $\s^2\times\R$, the $1$-parameter group is generated by a rotation of angle $2\ell(\lambda)$, whilst in $\h^2\times\R$ it is generated by a hyperbolic translation of length $2\ell(\lambda)$, where $\ell(\lambda) = -\int_0^{\pi/2}\nu_\lambda(h_0(t)) \df t$ is the length of the projection of $h_0^*$ to the slice $M^2(\epsilon)\times\{0\}$ (see Figure~\ref{fig:estrella}).

	\item We conjecture that the fundamental piece of the \textsc{cmc} $H$ surface in $\s^2\times \R$ is embedded for all values of $\lambda\in[0,\pi/2]$. Thus, for a suitable choice of $\lambda$, they will produce embedded \textsc{cmc} tori, different from the rotationally invariant one $S_0(H)$.
\end{enumerate}
\end{remark}

\begin{proof}\label{rmk:altura-monotona}
The previous reasoning ensures that we can extend the surface $\Sigma^*_\lambda$ to a complete \textsc{cmc} $H$ surface in $M^2(\epsilon)\times \R$ that we will denote by $S_\lambda(H)$, and it is clear that the extended surface is invariant by the mentioned group of isometries. We also know that $S_0(H)$ is the rotationally invariant torus (resp. cylinder) and $S_{\pi/2}(H)$ is the \textsc{cmc} sphere in $\s^2\times \R$ (resp.\ $\h^2 \times \R$).

In the case of $\h^2 \times \R$, since the minimal surface $\Sigma_\lambda$ is a graph over a convex domain (see Proposition~\ref{prop:existencia-Plateau-gamma-lambda}), the conjugate surface $\Sigma_\lambda^*$ is a graph over a certain domain of $\h^2$ by the generalized Krust's theorem given by~\cite{CH13}. Moreover, we claim that the surface $\Sigma_\lambda^*$ is contained in the region limited by the vertical planes $P_0$, $P_1$ and $P_2$ and the slice $\h^2 \times \{0\}$, which is equivalent to prove that the curve $v^*$ is contained in such region. This follows from the fact that the length of $v^*$ does not depend on $\lambda$, and its geodesic curvature $\kappa_g$ (computed as a curve of the slice $\h^2 \times \{0\}$ with respect to its normal vector field pointing outside the domain of the graph) satisfies the lower bound $\kappa_g \geq -(1+4H^2)/4H$ (i.e., it is bounded by the geodesic curvature of the equator of the \textsc{cmc} sphere of the same mean curvature, see~\cite[Theorem 3.3]{Man10}). Reasoning by contradiction, if the curve $v^*$ were not contained in the aforementioned region, it can be shown that the estimate above forces the length of the curve $v^*$ to be bigger than it is allowed to be. Finally, the complete surface $S_\lambda(H)$ is  embedded since the reflected fundamental regions do not intersect each other.

	On the other hand, we analyze the case of $\s^2 \times \R$, focusing on the length $a(\lambda) = -\int_{-\pi/4}^\lambda \nu_\lambda(h_2(t))\df t$ of the projection to $\s^2\times \{0\}$ of the curve $h_2^*$ (see Figure~\ref{fig:estrella}). The function $a(\lambda)$ is strictly increasing by Lemma~\ref{lm:crecimiento-angulo-frontera}.(iv) and takes all the values in between the corresponding length of the rotationally torus $\arctan(1/2H)$ and the sphere $2\arctan(1/2H)$. Now, if $H < 1/2$, then there exists a unique $\lambda^*$ such that $a(\lambda^*)=\pi/2$, so $a(\lambda)>\pi/2$ for all $\lambda\in[\lambda^*,\pi/2[$. For these values of $\lambda$ the complete surface $S_\lambda(H)$ has a self-intersection around the north pole of the sphere (we have consider $\Pi(h_0^*)$ as the equator in $\s^2 \times \{0\}$), so it is not embedded.

It is also clear that the maximum height must be attained at a point with $\nu_\lambda = -1$. Lemma~\ref{lm:crecimiento-angulo-frontera} ensures that this only happens at $h_0^*(0)$ and $h_0^*(\tfrac{\pi}{2})$. In the case of the horizontal cylinder, both points are at the same height and in the case of the sphere is trivial that the maximum height is attained at $h_0^*(\tfrac{\pi}{2})$. We are going to prove that, for $0<\lambda<\frac{\pi}{2}$, the maximum height is attained at $h_0^*(\tfrac{\pi}{2})$. 

Let us consider $h^*_i$ for $i\in\{0,1,2\}$ and write $h_i^*=(\beta_i,r_i)\in M^2(\epsilon)\times\R$. As $h_i^*$ is contained in a vertical plane which $\Sigma_\lambda^*$ meets orthogonally, it is easy to check that $|r'_i(t)|^2=1-\nu_\lambda(h^*_i(t))^2$ and $\|\beta_i'(t)\|=-\nu_\lambda(h_i^*(t))$. Since the angle function does not take the value $-1$ in the interior of $h_i^*$, we deduce that $r_i'\neq0$ along $h^*_i$, i.e., the height function is strictly monotonic along $h_i^*$ for $i\in\{0, 1, 2\}$. In particular, we deduce that the points $h_0^*(0)$ and $h_0^*(\frac{\pi}{2})$ does not have the same height. Moreover, taking this into account, the height $\mu_{\pi/2}(\lambda)$ of the point $h_0^*(\frac{\pi}{2})$ in $\Sigma_\lambda^*$ is given by
\[
\mu_{\pi/2}(\lambda) = \int_{-\pi/4}^{\lambda/2} \sqrt{1 - \nu_\lambda\bigl( h_2^*(t)\bigr)^2} \df t.
\]
In particular, $\mu_{\pi/2}(\lambda)$ is a continuous function of $\lambda$. Now, in view of Lemma~\ref{lm:crecimiento-angulo-frontera}, if $\lambda_1 < \lambda_2$, then
\[
\begin{split}
\mu_{\pi/2}(\lambda_1) &= \int_{-\pi/4}^{\lambda_1/2} \sqrt{1 - \nu_{\lambda_1}\bigl( h_2^*(t)\bigr)^ 2} \df t <\int_{-\pi/4}^{\lambda_2/2} \sqrt{1 - \nu_{\lambda_1}\bigl( h_2^*(t)\bigr)^ 2} <\\
&< \int_{-\pi/4}^{\lambda_2/2} \sqrt{1 - \nu_{\lambda_2}\bigl( h_2^*(t)\bigr)^2} \df t = \mu_{\pi/2}(\lambda_2).
\end{split}
\]
Hence, the height of the point $h_0^*(\tfrac{\pi}{2})$ is strictly increasing in $\lambda$. Finally, we get that $\mu_{\pi/2}(\lambda) \in\ ]\mu_{\pi/2}(0), \mu_{\pi/2}(\pi/2)[$ for every $\lambda \in\ ]0, \frac{\pi}{2}[$.

On the other hand, a similar argument shows that $\mu_0(\lambda)$, i.e.\ the height of the point $h_0^*(0)$ in $\Sigma_\lambda^*$, is a continuous strictly decreasing function of $\lambda$ and so $\mu_0(\lambda) \in\ ]0,\mu_0(0)[\ =\ ]0, \mu_{\pi/2}(0)[$ for $0 < \lambda < \frac{\pi}{2}$. Hence, the maximum height is attained at $h_0^*(\frac{\pi}{2})$ and it is between the height of the upper half of the horizontal cylinder and the height of the hemisphere.
\end{proof}

The properties shown in the proof allow us to make a quite precise depiction of the polygon $\Gamma_\lambda^*$, as can be seen in Figure \ref{fig:estrella}. It is important to observe that no information is obtained about $v^*$ apart from the fact that it is contained in a horizontal plane so the representation may not be exact.

\begin{figure}[htb]
 \centering
 \begin{minipage}[c]{0.5\textwidth}
 \includegraphics[width=\textwidth]{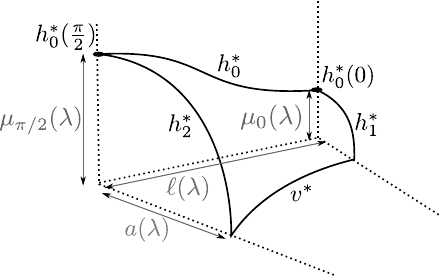}
\end{minipage}\quad\quad
\begin{minipage}[c]{0.4\textwidth}
 \includegraphics[width=\textwidth]{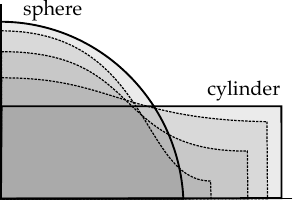}
\end{minipage} 
 \caption{Representation of the polygon $\Gamma_\lambda^*$ (left), where the dotted lines represent geodesics in $M^2(\epsilon)\times\R$, and a sketch of the profiles of the projections of $\Sigma^*_\lambda$ to the vertical plane containing $h_0^*$ (right). Note that the height of the cylinder is a half of the height of the sphere}\label{fig:estrella}
\end{figure}

Aledo, Espinar and Gálvez proved in~\cite{AEG08} that if $\Sigma\subseteq M^2(\epsilon)\times\R$ is a constant mean curvature $H>0$ graph over a compact open domain, with $4H^2+\epsilon>0$, whose boundary lies in the slice $M^2(\epsilon)\times\{0\}$, then $\Sigma$ can reach at most the height of the hemisphere, and equality holds if, only if, the surface is a rotationally invariant hemisphere. The construction above provides examples where the height varies between the height of the horizontal cylinder and the sphere. Those examples are not compact in general but in the case of $\s^2\times\R$ we get some compact ones.

\begin{corollary}
The family $\{S_\lambda(H):\lambda\in\; ]0, \pi/2[\}$ provides many compact constant mean curvature $H > 0$ in $\s^2\times\R$, whose maximum heights are dense between the height of the horizontal cylinder and the sphere. 
\end{corollary}

\begin{proof}
Let us consider the function $\ell:[0,\frac{\pi}{2}]\rightarrow\R$ defined in Remark~\ref{rmk:unduloids}. This function is nothing but the length of the geodesic segment $\Pi(h^*_0)\subseteq M^2(\epsilon)$ (see Figure~\ref{fig:estrella}). Due to Lemma~\ref{lm:crecimiento-angulo-frontera}.(i) this function is continuous and strictly decreasing. Now, it is clear that, by successively reflecting the piece $\Sigma_\lambda^*$, the obtained surface is compact if, and only if, $\ell(\lambda)$ is a rational multiple of $\pi$ so the corollary is proved.
\end{proof}

\begin{remark}\label{rmk:compact-cmc-1-loop-close}
In fact, by analyzing more deeply the arguments used in the proof, it can be shown that, for $H\geq 1/2$, it is possible to choose the parameter $\lambda$ so the resulting compact surface closes the first time it goes all the way round the equator. Indeed, the suitable choice for $\lambda$ is that one for which $\ell(\lambda)=\frac{\pi}{k}$ for some integer $k\geq 2$.
\end{remark}

\section{Constant mean curvature $1/2$ surfaces in $\h^2\times\R$}
\label{sec:cmc-1/2-H2xR}

This last section is devoted to construct \textsc{cmc} $1/2$ surfaces in $\h^2\times\R$ which have the symmetries of a tessellation of $\h^2$ by regular polygons. As we pointed out in the introduction, these surfaces arise from conjugation of minimal surfaces in the Heisenberg space. Recall that \textsc{cmc} $0 < H < 1/2$ surfaces in $\h^2 \times \R$ are obtained by the same procedure from minimal surfaces in the simplectic group $\Sl$, though this case is not considered in this paper. Let us introduce some notation to study that problem.

\begin{lemma}\label{lemma:teselacion}
Given $m,k\in\N$,  there exists a tessellation of $\h^2$ by regular $m$-gons such that  $k$ of them meet at each vertex if, and only if,
\[\frac{1}{m}+\frac{1}{k}<\frac{1}{2}.\]
Furthermore, such a tessellation is unique up to an isometry of $\h^2$. We will call it a $(m,k)$-tessellation.
\end{lemma}

Observe that each polygon in such a tessellation can be triangulated in $2m$ triangles whose angles are $\frac{\pi}{k}$, $\frac{\pi}{2}$ and $\frac{\pi}{m}$. Hence, we will construct a \textsc{cmc} 1/2 piece in the product of the triangle and the real line which will orthogonally meet the vertical planes passing through the sides of the triangles, as well as the slices where the triangle lies in (see Figure~\ref{fig:tesela}). This will be achieved by choosing an appropriate geodesic polygon in $\Nil$ and using the conjugate Plateau construction.

Given $\ell>0$ and $0<\alpha<\frac{\pi}{2}$, let us consider the geodesic polygon in $\Nil=\R^3$ given by
\[
\begin{aligned}
h_0(t) &= (\ell,t\ell\cot\alpha,\tfrac{1}{2}t\ell^2\cot\alpha), & t &\in \left[0,1\right] ,\\ 
h_1(t) &= (t\ell ,0,0),  &t &\in \left[0, 1\right],\\
h_2(t) &= (t\ell,t\ell\cot\alpha,\tfrac{1}{2}\ell^2\cot\alpha), & t &\in \left[0,1\right],\\
v(t) &= (0,0,\tfrac{1}{2}t\ell^2\cot\alpha), &t&\in \left[0, 1\right].
\end{aligned}
\]
This polygon is nothing but the horizontal lift of a triangle whose angles are $\frac{\pi}{2}$, $\alpha$ and $\frac{\pi}{2}-\alpha$, where a vertical segment $v$ has been placed in the vertex with angle $\frac{\pi}{2}-\alpha$. The boundary of the polygon lies in the boundary of a mean-convex body, namely, the lift of the whole triangle in $\R^2$ whose boundary consists of three vertical (and thus minimal) planes (cf.\ Subsection~\ref{subsec:heisenberg}).

\begin{figure}[htbp]
\centering
\includegraphics{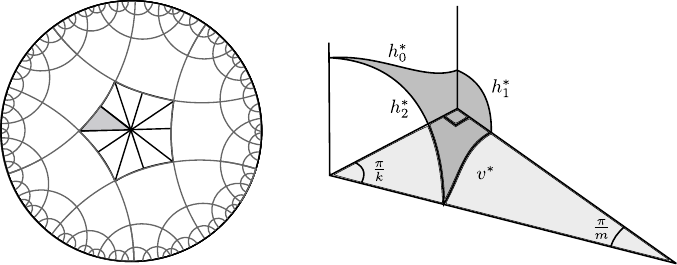}
\caption{A $(m,k)$-tessellation of $\h^2$  for $m=5$ and $k=4$ (left) and the fundamental piece of the \textsc{cmc} surface (right), that fits in the shaded triangle. }\label{fig:tesela}
\end{figure}
Following the ideas developed in previous sections, it is possible to solve the Plateau problem for this polygon in $\Nil$ and we obtain a minimal quadrilateral which is also a graph over an Euclidean right triangle in the plane $z=0$. The angle determined by $h_0$ and $h_2$ is equal to $\alpha$ whereas the three remaining angles are equal to $\frac{\pi}{2}$. If we apply the Daniel correspondence, it provides a \textsc{cmc} $1/2$ quadrilateral in $\h^2\times\R$ for which the corresponding edges $h_i^*$, $i\in\{0,1,2\}$, lie in vertical planes $V_i$ in such a way that $V_0$ and $V_2$ meet with angle $\alpha$ and $V_0$ and $V_1$ are orthogonal.

Nevertheless, we would like the planes $V_0$, $V_1$ and $V_2$ to fit the triangle constructed in the tessellation of $\h^2$, for what we choose $\alpha=\frac{\pi}{k}$. Next step will consist in estimating the length of $\Pi(h_0^*)\subset\h^2$ when varying the parameter $\ell>0$. Notice that that length is the integral of the angle function of $\Sigma^*$ along $h_0^*$ and the angle function is preserved by the correspondence.

\begin{lemma}\label{lemma:tessellation}
In the construction above, the length of the segment $\Pi(h_0^*)\subset\h^2$ diverges when $\ell\to\infty$ (for fixed $\alpha\, \in]0,\pi/2[$).
\end{lemma}

\begin{proof}
On the one hand, let us consider $P$ to be the surface given by the equation $z=\frac{1}{2}(2\ell-x)y$. Then $P$ is a minimal graph (it is congruent to $z=\frac{xy}{2}$ by an ambient isometry), contains the horizontal geodesics $h_0$ and $h_1$ and it is a horizontal surface (i.e., $\nu_P=-1$) along $h_0$. Moreover, a direct application of the maximum principle yields that $\Sigma$ lies above $P$ and hence the angle function of $\Sigma$ satisfies $-1<\nu<0$ in the interior of $h_0$. 

On the other hand, let $Q\subset\Nil$ be the image of the plane $\{z = 0\}$ by a translation that sends the origin to the intersection of $h_0$ and $h_2$. Again the maximum principle guarantees that $\Sigma$ lies in a region bounded by $P$ and $Q$. Since $\nu>-1$, it is possible to compare the angle functions of $\Sigma$ and $Q$ along $h_0$, concluding that $-1<\nu\leq\nu_Q<0$ in the interior of $h_0$. Moreover, the integral of $\nu_Q$ along $h_0$ can be computed explicitly to show that it diverges when $\ell\to\infty$, which forces the integral of $\nu$ along $h_0$ (i.e., the length of $\Pi(h_0^*)\subset\h^2$) to diverge, and the proof is finished.
\end{proof}

As the length of $\Pi(h_0^*)\subset\h^2$ is a continuous function of $\ell$, we deduce that it takes all positive values. In particular Lemma~\ref{lemma:tessellation} implies the construction of the fundamental piece for any regular tessellation. 

\begin{theorem}\label{thm:tiling-surface}
Given a $(m,k)$-tessellation of $\h^2\times\{0\}$, there exists a constant mean curvature $1/2$ bi-multigraph in $\h^2\times\R$  with boun\-ded height, invariant under any isometry of $\h^2\times\R$ which preserves the tessellation. 
\end{theorem}

\smallskip
\begin{remark}
The fundamental piece of the surface is known to be embedded by the results in~\cite{CH13} but some issues appear to conclude that the whole surface remains embedded after reflecting such a piece. Our conjecture is that all constructed surfaces are properly embedded bigraphs.
\end{remark}

Finally, we will apply this result to study \textsc{cmc} $1/2$ surfaces when considering some quotients of the hyperbolic plane rather than the plane itself. Let us take a regular $2m$-gon ($m\geq 2$) in the hyperbolic plane and suppose we can identify some of its sides in pairs to obtain a compact surface in such a way there exists a positive integer $k\geq 3$ such that the vertices of the polygon are identified in classes of $k$ elements each. We will call it a \emph{regular gluing pattern}. Then, the surface and the identifications can be carried out in a $(2m,k)$-tessellation of the hyperbolic plane (observe that $\frac{1}{2m}+\frac{1}{k}<\frac{1}{2}$ so Lemma \ref{lemma:tessellation} can be applied), which shows a way to endow the resulting surface with a metric of constant curvature $-1$ whose universal Riemannian cover is the hyperbolic plane (see \cite[Section 1.3]{Thu} for a more detailed description). Gauss-Bonnet formula implies that the resulting surfaces have negative Euler characteristic. Let us illustrate this situation with some examples:
\begin{itemize}
 \item Given $g\geq 2$, consider the gluing pattern in a $4g$-gon defined by
\[
a_1b_1a_1^{-1}b_1^{-1}a_2b_2a_2^{-1}b_2^{-1}\cdots a_gb_ga_g^{-1}b_g^{-1}.
\]
All vertices are identified together so it leads to a $(4g,4g)$-te\-sse\-lla\-tion. The obtained surface is a genus $g$ orientable surface.
 \item For $g\geq 3$, consider now the gluing pattern in a $2g$-gon given by
\[
a_1a_1a_2a_2\cdots a_ga_g.
\]
All vertices are identified together again so it leads to a $(2g,2g)$-tessellation. The quotient is  a non-orientable genus $g$ surface.
\end{itemize}
The conditions on $g$ are the geometric restrictions on a surface of genus $g$ to have negative Euler characteristic. Other identifications give rise to the same topological surfaces, but not isometric to these ones.

Suppose now that we have a regular gluing pattern to which we associate a $(2m,k)$-tessellation. By applying Gauss-Bonnet formula to a regular $2m$-gon $P_{(2m,k)}$ in $\h^2$ with interior angles equal to $2\pi/k$, we get
\[
	\int_{P_{(2m,k)}}K_{\h^2}=2\pi\left(1+\frac{2m}{k}-m\right).
\]
If a compact surface $M$ is obtained from $P_{(2m,k)}$ when identifying some of its edges, then it has Euler characteristic $\chi(M)=1+\frac{2m}{k}-m$. Let us now consider the surface $\Sigma^*_{(2m,k)}$ given by Theorem \ref{thm:tiling-surface}. As every symmetry of the tiling is also a symmetry of the surface, its edges can be identified in the same way as those of $P_{(2m,k)}$ when constructing $M$ and it provides a compact \textsc{cmc} $1/2$ surface in the quotient space $M\times\R$. Finally, we will compute the Euler characteristic of $\hat{\Sigma}^*_{(2m,k)}$, the quotient surfaces. Since it consists of $8m$ pieces, each of which coming from the piece $\Sigma$ constructed in $\Nil$, which satisfies $\int_\Sigma K_\Sigma=\pi/k-\pi/2$ (notice that $\Sigma$ is a quadrilateral whose angles are $\pi/2$, $\pi/2$, $\pi/2$ and $\pi/k$), we obtain
\[	\int_{\hat{\Sigma}^*_{(2m,k)}}K_{\hat{\Sigma}^*_{(2m,k)}}=4\pi\left(\frac{2m}{k}-m\right),
\]
so $\chi(\hat{\Sigma}^*_{(2m,k)})=2(\frac{2m}{k}-m)$. 

Now observe that the gluing pattern in $M$ induces other pattern in $\Sigma^*_{(2m,k)}$, which can be seen as a certain polygon whose sides have been identified in pairs and it is easy to realize that $\hat{\Sigma}^*_{(2m,k)}$ is orientable if, and only if, $M$ is orientable. For orientable surfaces, the genus and Euler characteristic satisfy $\chi=2(1-g)$, whereas $\chi=2-g$ is satisfied for the non-orientable case. From this, the following result follows.

\begin{corollary}
Let $M$ a compact Riemannian surface with negative Euler characteristic and constant curvature $-1$ which can be realized by a regular gluing pattern. Then the construction above induces a compact constant mean curvature $1/2$ bi-multigraph $\Sigma$ immersed in $M\times\R$ satisfying:
\begin{enumerate}[(i)]
 \item $\Sigma$ is orientable if, and only if, $M$ is orientable.
 \item If $M$ has genus $g$, then $\Sigma$ has genus $2g$.
\end{enumerate}
\end{corollary}

\end{document}